\documentclass[11pt,reqno,a4paper]{amsart}
\usepackage{amsmath,amsthm,verbatim,amscd,amssymb,setspace,esint,mathtools}
\usepackage{exscale,color}
\usepackage{tensor}

\usepackage{tikz-cd}

\usepackage{mathrsfs}

\usepackage{enumitem}

\usepackage{hyperref}

\renewcommand{\epsilon}{\varepsilon}
\newcommand{\N}{\mathbb{N}}

\newcommand{\R}{\mathbb{R}}

\newcounter{mtheorem}
\newtheorem{mtheorem}[mtheorem]{Theorem}

\newcommand{\p}{\partial}

\newcommand{\Ric}{\operatorname{Ric}}

\newcommand{\Rm}{\operatorname{Rm}}

\newtheoremstyle{fancy}{}{}{\itshape}{}{\textbf\bgroup}{.\egroup}{ }{}
\newtheoremstyle{fancy2}{}{}{\rm}{}{\textbf\bgroup}{.\egroup}{ }{}

\theoremstyle{fancy}
\newtheorem{theorem}{Theorem}[section]
\newtheorem{lemma}[theorem]{Lemma}

\theoremstyle{fancy2}
\newtheorem{definition}[theorem]{Definition}

\newtheorem{claim}[theorem]{Claim}

\setlist{leftmargin=*}

\numberwithin{equation}{section}
\begin{document}
\title{A non-K\"ahler expanding Ricci soliton with a K\"ahler tangent cone at infinity}

\author{Richard H.~Bamler}
\address{Department of Mathematics, University of California Berkeley, CA 94720}
\email{rbamler@berkeley.edu}

\author{Eric Chen}
\address{Department of Mathematics, University of Illinois Urbana-Champaign, Urbana, IL 61801}
\email{ecchen@illinois.edu}

\author{Ronan J.~Conlon}
\address{Department of Mathematical Sciences, The University of Texas at Dallas, Richardson, TX 75080}
\email{ronan.conlon@utdallas.edu}

\date{\today}

\begin{abstract}
We construct an example of an asymptotically conical (AC) non-K\"ahler expanding gradient Ricci soliton that has a K\"ahler tangent cone at infinity. This yields an example of a K\"ahler cone that can be desingularised by a smooth 
AC expanding gradient Ricci soliton but not by a smooth AC expanding gradient K\"ahler--Ricci soliton.
\end{abstract}

\maketitle

\markboth{Richard H.~Bamler, Eric Chen, and Ronan J.~Conlon}{A non-K\"ahler expander with a K\"ahler tangent cone at infinity}

\section{Introduction}

\subsection{Overview}
A \emph{Ricci soliton} is a triple $(M,\,g,\,X)$, where $M$ is a Riemannian manifold endowed with a complete Riemannian metric $g$
and a complete vector field $X$, such that
\begin{equation}\label{hot}
\Ric_{g}=\frac{1}{2}\mathcal{L}_{X}g+\frac{\lambda}{2}g
\end{equation}
for some $\lambda\in\mathbb{R}$. The vector field $X$ is called the
\emph{soliton vector field}. If $X=\nabla^{g} f$ for some smooth real-valued function $f$ on $M$,
then we say that $(M,\,g,\,X)$ is \emph{gradient}. In this case, the soliton equation \eqref{hot}
becomes
$$\Ric_{g}=\nabla^{2}f+\frac{\lambda}{2}g,$$
and we call $f$ the \emph{soliton potential}. In the case of gradient Ricci solitons, the completeness of $X$ is guaranteed by the completeness of $g$
\cite{zhang12}.

If $g$ is K\"ahler and $X$ is real holomorphic, then we say that the triple $(M,\,g,\,X)$ is a \emph{K\"ahler--Ricci soliton} if
\begin{equation}\label{hotter}
\Ric_{g}=\frac{1}{2}\mathcal{L}_{X}g+\lambda g
\end{equation}
for some $\lambda\in\mathbb{R}$. Let $\omega$ denote the K\"ahler form of $g$. 
If $X=\nabla^{g} f$ for some smooth real-valued function $f$ on $M$,
then we say that $(M,\,g,\,X)$ is \emph{gradient} and call $f$ the \emph{soliton potential}. In this case, 
\eqref{hotter} may be rewritten as
\begin{equation*}
\rho_{\omega}=i\partial\bar{\partial}f+\lambda\omega,
\end{equation*}
where $\rho_{\omega}$ is the Ricci form of $\omega$. Clearly, given any K\"ahler--Ricci soliton $(M,\,g,\,X)$, the rescaling
$(M,\,2g,\,\frac{1}{2}X)$ defines a Ricci soliton. 

Finally, a Ricci soliton and a K\"ahler--Ricci soliton are called \emph{steady} if $\lambda=0$, \emph{expanding}
if $\lambda<0$, and \emph{shrinking} if $\lambda>0$ in \eqref{hot} and \eqref{hotter} respectively.
One can always normalise $\lambda$, when non-zero, to satisfy $|\lambda|=1$. We henceforth assume that this is the case. 

The study of Ricci solitons and their classification is important in the context of Riemannian
geometry. For example, they provide a natural generalisation of Einstein manifolds. Also, to each Ricci
soliton, one may associate a self-similar solution of the Ricci flow \cite[Lemma 2.4]{chowbook} and these are
candidates for singularity models of the flow. The same statement also holds for K\"ahler--Ricci solitons and the K\"ahler--Ricci flow; see Section \ref{finally} for more details. 
The difference in normalisations between \eqref{hot} and \eqref{hotter} is
consistent with the Ricci flow and K\"ahler--Ricci flow equations, respectively, when one takes this dynamic point of view.

It is known
that any complete expanding (respectively shrinking) gradient Ricci soliton whose curvature decays quadratically with derivatives (resp.~quadratically) along an end 
is asymptotically conical (AC) with a unique tangent cone with a smooth link along that end \cite{Che-Der, Chow, wangl, Siepmann}, which dynamically appears as
the initial (resp.~terminal) condition of the Ricci flow $g(t),\,t\geq0$, (resp.~$g(t),\,t\leq0$)
associated to the soliton, in the sense that $\lim_{t\to0^{+}}g(t)=g_{0}$ (resp.~$\lim_{t\to0^{-}}g(t)=g_{0}$) as a Gromov-Hausdorff limit \cite[Remark 1.5]{Che-Der} and where convergence is locally smooth away from the tip of the cone. Now, if the tangent cone of an AC shrinking gradient Ricci soliton is 
K\"ahler, then this property propagates backwards in time \cite{kotschwar}, i.e., the shrinking Ricci soliton is a shrinking K\"ahler--Ricci soliton. 
In this note, we present an example demonstrating that the same behaviour is not necessarily carried forward in time by an expanding gradient Ricci
soliton\footnote{For each $p\in(0,\,1)$, Cao \cite{cao2} constructs a complete AC expanding gradient K\"ahler--Ricci soliton 
on $\mathbb{C}^{2}$ with asymptotic cone $\left(\mathbb{C}^{2},\,i\partial\bar{\partial}\left(\frac{1}{p}|z|^{2p}\right)\right)$.
These expanding K\"ahler--Ricci solitons have strictly positive curvature operator on real $(1,\,1)$-forms 
\cite{Che-Zhu-Pos-Cur}, hence have non-negative curvature operator so that the same property holds true on the asymptotic
cone. (This can also be verified by a direct computation on the cone.)
In addition, the expanding gradient Ricci soliton of Deruelle emanating from the same cone \cite{Der-Smo-Pos-Cur-Con} has non-negative curvature operator.
By the uniqueness result \cite[Theorem 1.3]{Der-Smo-Pos-Cur-Con} of the same paper, these latter expanding Ricci solitons must coincide with the former.}.

More precisely, we give an example of a complete non-K\"ahler AC expanding gradient Ricci soliton with a K\"ahler tangent cone at infinity.
This yields a smooth Ricci flow emanating from a K\"ahler cone that is not a K\"ahler--Ricci flow. It has been shown recently that under certain conditions, a smooth K\"ahler--Ricci flow emanating from a K\"ahler cone is necessarily induced by a complete AC expanding gradient K\"ahler--Ricci soliton \cite{chan123, LChen}. In contrast, our result demonstrates that a smooth Ricci flow emanating from a K\"ahler cone need not be induced by a complete expanding gradient K\"ahler--Ricci soliton. 

\subsection{Result}
We say that an expanding Ricci soliton $(M,\,g)$ has ``quadratic curvature decay with derivatives'' if
\begin{equation*}
A_{k}(g):=\sup_{x\in M}|(\nabla^{g})^{k}\operatorname{Rm}_{g}|_{g}(x)d_{g}(p,\,x)^{2+k}<\infty\quad\textrm{for all $k\in\mathbb{N}_{0}$}, 
\end{equation*}
where $d_{g}(p,\,\cdot)$ denotes the distance to a fixed point $p\in M$ with respect to $g$. By \cite[Section 3.1]{cds}, 
this condition is equivalent to being asymptotically conical (AC) in a precise sense (cf.~Definition \ref{d:AC}). Our result is the following.
\begin{mtheorem}\label{thm_nonkahler}
There exists a real four-dimensional complete non-K\"ahler expanding gradient Ricci soliton with one end with quadratic curvature decay with derivatives (or equivalently, AC)
that has a K\"ahler tangent cone at infinity.
\end{mtheorem}

As one shall see in the proof, a concrete example of a K\"ahler cone in Theorem \ref{thm_nonkahler} that appears as the tangent cone of a complete AC
expanding Ricci soliton but not of a complete AC expanding K\"ahler--Ricci soliton is the complex cone $\mathbb{C}^{2}/\Gamma_{3,\,2}$
endowed with an appropriate K\"ahler cone metric that we denote by $g_{0}$. Here, $\Gamma_{3,\,2}$ is the subgroup of $U(2)$ generated by $\operatorname{diag}\left(e^{\frac{2\pi i}{3}},\,e^{\frac{4\pi i}{3}}\right)$.
 This complex cone does not have a smooth canonical model as its minimal model contains $(-2)$-curves, hence it cannot be desingularised by a complete AC expanding gradient K\"ahler--Ricci soliton \cite[Corollary B]{cds}. In particular, it provides an example of a K\"ahler cone that can be desingularised by a complete AC expanding gradient Ricci soliton (in light of Theorem \ref{thm_nonkahler}), but not by a complete AC expanding gradient K\"ahler--Ricci soliton. However, this K\"ahler cone can be desingularised by an orbifold AC expanding gradient K\"ahler--Ricci soliton. To see this, one lifts the aforementioned K\"ahler cone metric $g_{0}$ to $\mathbb{C}^{2}$ to get a $\Gamma_{3,\,2}$-invariant K\"ahler cone metric $\tilde{g}_{0}$ on $\mathbb{C}^{2}$. By \cite[Corollary B]{cds}, there exists a unique complete AC expanding gradient K\"ahler--Ricci soliton on $\mathbb{C}^{2}$ with tangent cone $\tilde{g}_{0}$. The uniqueness part of \cite[Corollary B]{cds} then implies that the 
$\Gamma_{3,\,2}$-invariance of $\tilde{g}_{0}$ is inherited by the expanding soliton. The soliton therefore descends to an orbifold expanding gradient K\"ahler--Ricci soliton on the quotient
$\mathbb{C}^{2}/\Gamma_{3,\,2}$ with tangent cone $g_{0}$. 

The proof of Theorem \ref{thm_nonkahler} combines the characterisation of complete
AC expanding K\"ahler--Ricci solitons given by \cite{cds} with local deformation properties of complete AC expanding Ricci solitons with invertible weighted Lichnerowicz operator given by \cite{bamler-chen}; recent work of Naff--Ozuch \cite{ozuch-naff} provides the invertibility of this operator in our setting. Our starting point is a suitable flat Riemannian cone admitting two distinct complex structures. By \cite{cds}, one of these K\"ahler cones appears as the tangent cone of a complete AC expanding gradient K\"ahler--Ricci soliton, whereas the
other one does not. We deform the Riemannian structure of the cone in a direction compatible with the second complex structure to obtain another K\"ahler cone. 
By a local deformation result of \cite{bamler-chen} (which can also be deduced from \cite{Der-Smo-Pos-Cur-Con}) together with the
 invertibility of the weighted Lichnerowicz operator for real four-dimensional expanding gradient K\"ahler--Ricci solitons \cite{ozuch-naff}, 
 there exists a complete expanding gradient Ricci soliton asymptotic to this deformed Riemannian cone. However, using the structure theorem of \cite{cds} 
for expanding gradient K\"ahler--Ricci solitons together with properties of complex structures on real four-dimensional K\"ahler cones, we see that this expanding Ricci soliton cannot be K\"ahler.

\subsection{Acknowledgements} This material is based upon work supported by the National Science Foundation under Grant No.~DMS-1928930, while the authors were in residence
at the Simons Laufer Mathematical Sciences Institute (formerly MSRI) in Berkeley, California, during the Fall 2024 semester. They wish to thank SLMath for their excellent working conditions and hospitality during this time. They also wish to thank Alix Deruelle, Yuji Odaka, Jeff Viaclovsky, and Junsheng Zhang for useful discussions. 

Part of this work was also carried out when the third author was in residence at the Simons Centre for Geometry and Physics in January 2026 
for the program entitled \emph{Einstein 4-Manifolds and Gravitational Instantons}. He wishes to thank the centre for their excellent working conditions and hospitality during this time.

The first author is supported by NSF grant DMS-2204364. The last two authors are supported by Simons Travel Grants. 

\section{Preliminaries}\label{sec_preliminaries}

\subsection{Riemannian cones} For us, the definition of a Riemannian cone will take the following form.

\begin{definition}\label{cone}
Let $(S, g_{S})$
be a compact connected Riemannian manifold. The \emph{Riemannian cone} $C_{0}$ with \emph{link} $S$ is defined to be $\R^+ \times S$ with metric $g_0 = dr^2 \oplus r^2g_{S}$ up to isometry. The radius function $r$ is then characterized intrinsically as the distance from the apex in the metric completion.
\end{definition}

Suppose that we are given a Riemannian cone $(C_0,g_{0})$ as above. Let $(r,x)$ be polar coordinates on $C_{0}$, where $x\in S$, and for $t>0$, define a map
$$\nu_{t}: [1,2]\times S \ni (r,x) \mapsto (tr,x) \in [t,2t] \times S.$$ One checks that $\nu_{t}^{*}(g_{0})=t^{2}g_{0}$ and $\nu^{*}_{t}\circ\nabla^{g_0}=\nabla^{g_0}\circ\nu_{t}^{*}$, where $\nabla^{g_0}$ is the  Levi-Civita connection of $g_{0}$. From this, we deduce

\begin{lemma}\label{simple321}
Suppose that $\alpha\in\Gamma((TC_0)^{\otimes p}\otimes (T^{*}C_0)^{\otimes q})$ satisfies $\nu_{t}^{*}(\alpha)=t^{k}\alpha$ for every $t>0$ for some $k\in\R$. Then $|(\nabla^{g_0})^{\ell}\alpha|_{g_{0}}=O(r^{k+p-q-\ell})$ for all $\ell\in\N_0$.
\end{lemma}

We shall say that ``$\alpha=O(r^{\lambda})$ with $g_{0}$-derivatives'' whenever $|(\nabla^{g_0})^{k}\alpha|_{g_{0}}=O(r^{\lambda-k})$ for every $k \in \N_0$.
We will then also say that $\alpha$ has ``rate at most $\lambda$'', or sometimes, for simplicity, ``rate $\lambda$'', although it should be understood that (at least when $\alpha$ is purely polynomially behaved and does not contain any $\log$ terms) the rate of $\alpha$ is really the infimum of all $\lambda$ for which this holds.

\subsection{K{\"a}hler cones} Boyer-Galicki \cite{book:Boyer} is a comprehensive reference here.

\begin{definition}A \emph{K{\"a}hler cone} is a Riemannian cone $(C_0,g_0)$ such that $g_0$ is K{\"a}hler, together with a choice of $g_0$-parallel complex structure $J_0$. This will in fact often be unique up to sign. We then have a K{\"a}hler form $\omega_0(X,Y) = g_0(J_0X,Y)$, and $\omega_0 = \frac{i}{2}\p\bar{\p} r^2$ with respect to $J_0$.
\end{definition}

The vector field $r\partial_{r}$ on a K\"ahler cone is real holomorphic, and $J_{0}r\partial_r$ is real holomorphic and Killing \cite[Appendix A]{MSY}. This latter vector field is known as the \emph{Reeb field}. 

\subsection{Type II deformations of K\"ahler cones}\label{type2}

Given a K\"ahler cone $C_{0}$ of complex dimension $n$ with radius function $r$ and K\"ahler form $\omega_{0}=\frac{i}{2}\partial\bar{\partial}r^{2}$, it is true that
\begin{equation*}
\omega_{0}=rdr\wedge\eta +\frac{1}{2}r^{2}d\eta,
\end{equation*}
 where $\eta=i(\bar{\partial}-\partial)\log r$. One may deform the K\"ahler cone metric as follows. Let $J_{0}$ denote the complex structure on $C_{0}$. Then take any smooth real-valued function $\varphi$ on $C_{0}$ with \nolinebreak $\mathcal{L}_{r\partial_{r}}\varphi=\mathcal{L}_{J_{0}r\partial_{r}}\varphi=0$ such that $\tilde{\omega}_{0}=\frac{i}{2}\partial\bar{\partial}(r^{2}e^{2\varphi})>0$. The form
$\tilde{\omega}_{0}$ will define a new K\"ahler cone metric on $C_{0}$ with radius function $\tilde{r}:=re^{\varphi}$. The radial vector field $r\partial_{r}=\tilde{r}\partial_{\tilde{r}}$
and link of the cone $\{\tilde{r}=1\}$ remain unchanged. Let $\tilde{\eta}=i(\bar{\partial}-\partial)\log\tilde{r}$. Then $\tilde{\omega}_{0}$ may be written as
\begin{equation*}
\begin{split}
\tilde{\omega}_{0}&=\frac{i}{2}\partial\bar{\partial}(r^{2}e^{2\varphi})=\tilde{r}d\tilde{r}\wedge\tilde{\eta} +\frac{1}{2}\tilde{r}^{2}d\tilde{\eta}\\
&=ie^{2\varphi}(r\partial\bar{\partial}r+\partial r\wedge\bar{\partial}r+2r(\partial r\wedge\bar{\partial}\varphi+\partial\varphi\wedge\bar{\partial}r)+2r^{2}\partial\varphi\wedge\bar{\partial}\varphi+r^{2}\partial\bar{\partial}\varphi)\\
&=e^{2\varphi}\omega_{0}+2ire^{2\varphi}(\partial r\wedge\bar{\partial}\varphi+\partial\varphi\wedge\bar{\partial}r+r\partial\varphi\wedge\bar{\partial}\varphi
+\frac{1}{2}r\partial\bar{\partial}\varphi).
\end{split}
\end{equation*}
A deformation of this type is called a ``deformation of type II''; see for example \cite[Section 7.5.1]{book:Boyer} and \cite[Proposition 4.2]{futaki} for more details.

\subsection{Asymptotically conical expanding Ricci solitons}

``Asymptotically conical'' for expanding Ricci solitons means the following.
\begin{definition}\label{d:AC}
Let $(M,\,g,\,X)$ be a complete expanding Ricci soliton and let $(C_{0},\,g_0)$ be a Riemannian cone. We call $(M,\,g,\,X)$ an \emph{asymptotically conical} (AC) \emph{expanding Ricci soliton} with tangent cone $C_{0}$ if there exists a diffeomorphism $\Phi:C_{0}\setminus K \to M \setminus K'$ with $K,K'$ compact, such that $\Phi^*g - g_0 = O(r^{-\epsilon})$ with $g_0$-derivatives for some $\epsilon > 0$. 
\end{definition}
\noindent In particular, the curvature $\operatorname{Rm}_{g}$ of a complete AC expanding Ricci soliton $(M,\,g)$ satisfies
\begin{equation}\label{eric}
A_{k}(g):=\sup_{x\in M}|(\nabla^{g})^{k}\operatorname{Rm}_{g}|_{g}(x)d_{g}(p,\,x)^{2+k}<\infty\quad\textrm{for all $k\in\mathbb{N}_{0}$},
\end{equation}
where $d_{g}(p,\,\cdot)$ denotes the distance to a fixed point $p\in M$ with respect to $g$. On the other hand, by \cite[Theorem 3.8]{cds},
any complete expanding gradient Ricci soliton $(M,\,g)$ satisfying \eqref{eric} is an AC expanding Ricci soliton with $\varepsilon=2$ 
and $d\Phi(r\partial_{r})=2X$ in Definition \ref{d:AC}.

\subsection{Asymptotically conical expanding K\"ahler--Ricci solitons}\label{founders}
In the K\"ahler setting, ``asymptotically conical'' for expanding solitons means the following.
\begin{definition}\label{d:ACK}
Let $(M,\,g,\,X)$ be a complete expanding K\"ahler--Ricci soliton with complex structure $J$ and let $(C_{0},g_0)$ be a K\"ahler cone with a choice of $g_0$-parallel complex structure $J_0$.
We call $(M,\,g,\,X)$ an \emph{asymptotically conical} (AC) \emph{expanding K\"ahler--Ricci soliton} with tangent cone $C_{0}$ if there exists a diffeomorphism $\Phi: C_{0}\setminus K \to M \setminus K'$ with $K,K'$ compact, such that $\Phi^*g - g_0 = O(r^{-\epsilon})$ with $g_0$-derivatives and $\Phi^*J - J_0 = O(r^{-\epsilon})$ with $g_0$-derivatives for some $\epsilon > 0$. 
\end{definition}
\noindent  By \cite[Theorem A]{cds}, one may assume that $\Phi$ is a biholomorphism with respect to which $\varepsilon=2$ and $d\Phi(r\partial_{r})=X$.
In addition, after replacing $(g,\,X)$ and $g_{0}$ with $(2g,\,\frac{1}{2}X)$ and $\psi^{*}g_{0}$ respectively, where 
$\psi:C_{0}\rightarrow C_{0}$ is the map that sends $r\mapsto\sqrt{2}r$, by abuse of notation we can consider $(M,\,g,\,X)$ an AC expanding Ricci soliton with tangent cone $(C_{0},\,g_{0})$.

\subsection{Expanding solitons and the Ricci flow}\label{finally}

A complete expanding gradient Ricci (respectively K\"ahler--Ricci) soliton $(M,\,g,\,X)$ with soliton potential $f$ defines a homothetically expanding 
solution of the Ricci (resp.~K\"ahler--Ricci) flow in the following way. Set $$g(t):=t\varphi_{t}^{*}g,\quad t>0,$$
where $\varphi_{t}$ is the family of diffeomorphisms (resp.~biholomorphisms) generated by the gradient vector field $-\frac{1}{t}X$ (resp.~$-\frac{1}{2t}X$)
with $\varphi_{1}=\operatorname{id}$, i.e.,
\begin{equation*}
\frac{\partial\varphi_{t}}{\partial t}(x)=-\frac{\nabla^g f(\varphi_{t}(x))}{t},\quad\varphi_{1}=\operatorname{id},\qquad\left(\textrm{resp.}\quad\frac{\partial\varphi_{t}}{\partial t}(x)=-\frac{\nabla^g f(\varphi_{t}(x))}{2t},\quad\varphi_{1}=\operatorname{id}\right).
\end{equation*}
Then $\partial_t g(t)=-2\operatorname{Ric}_{g(t)}$ (resp.~$\partial_t g(t)=-\operatorname{Ric}_{g(t)}$) for $t>0$ and $g(1)=g$. 

Let $(M,\,g,\,X)$ be a complete AC expanding gradient Ricci soliton with tangent cone $(C_{0},\,g_{0})$ with radius function $r$.
Then in light of Definition \ref{d:AC} and the comments thereafter, there exists a diffeomorphism $\Phi: C\setminus K \to M \setminus K'$ with $K,K'$ compact, and $d\Phi(r\partial_{r})=2X$
with respect to which 
\begin{equation*}
|(\nabla^{g_{0}})^k(\Phi^{*}g-g_{0})|_{g_0}\leq C_{k}r^{-2-k}\qquad\textrm{for all $k\in\mathbb{N}_{0}$}.
\end{equation*}
The corresponding Ricci flow $g(t),\,t>0,$ then satisfies 
\begin{equation}\label{flowbaby}
|(\nabla^{g_{0}})^k(\Phi^{*}g(t)-g_{0})|_{g_0}(x)\leq C_{k}tr(x)^{-2-k}\qquad\textrm{for all $k\in\mathbb{N}_{0}$},
\end{equation}
at all points $x$ with $r(x)$ large. This follows from the computations on \cite[p.301]{cds}.
Thus, it is clear that $\lim_{t\to0^{+}}\Phi^{*}g(t)=g_{0}$ locally smoothly on the asymptotic cone $C_{0}$. 
The same statement also holds for complete AC expanding gradient K\"ahler--Ricci solitons, except that $d\Phi(r\partial_r)=X$.

\section{Proof of Theorem \ref{thm_nonkahler}}\label{sec_thm}

Consider $\mathbb{C}^{2}$ endowed with the standard complex structure 
 $J_{0}:=\left(\begin{array}{cc}
                                        i & 0 \\
                                        0 & i  \\
                                      \end{array}
                                    \right)$
in each tangent space, and the action of the group $\Gamma_{p,\,q}$ on $\mathbb{C}^{2}$ generated by the matrix $\left(\begin{array}{cc}
                                        e^{\frac{2\pi i}{p}} & 0 \\
                                        0 & e^{\frac{2\pi i q}{p} } \\
                                      \end{array}\right),$
where $p$ and $q$ are coprime integers with $p>q>0$. This group acts freely on $\mathbb{C}^{2}\setminus\{0\}$ and preserves the flat metric on $\mathbb{C}^{2}$. Indeed, this group is a finite subgroup of $U(2)$. We consider the K\"ahler cones $\widehat{C}_{p,\,q}:=(\mathbb{C}^{2},\,J_{0})/\Gamma_{p,\,q}$ endowed with the flat metric $g_{0}$. These cones are isometric to the flat Riemannian cone $C(L(p;q))$ over the lens space $L(p;q)$. 
As is well known, $L(p;q)$ is isometric to $L(p;p-q)$ so that the corresponding Riemannian cones $(\widehat{C}_{p,\,q},\,g_{0})$ and $(\widehat{C}_{p,\,p-q},\,g_{0})$ respectively
are isometric. Indeed, an (orientation-reversing) isometry is given by the map $\phi:\widehat{C}_{p,\,q}\to\widehat{C}_{p,\,p-q},\,\phi(z_{1},\,z_{2})=(z_{1},\,\bar{z}_{2})$.
The pullback complex structure $\phi^{*}J_{0}$ is then given by $\phi^{*}J_{0}=\left(\begin{array}{cc}
                                        i & 0 \\
                                        0 & -i  \\
                                      \end{array}
                                    \right).$
On $\widehat{C}_{p,\,p-q}$, $g_{0}$ is K\"ahler with K\"ahler form given by
$\omega_{0}:=\frac{i}{2}\partial\bar{\partial}r^{2}$, where $r=|z|$ for $z\in\mathbb{C}^{2}$. 

Recalling the conditions under which two lens spaces are diffeomorphic, 
choose $p>q>0$ coprime integers such that there is no solution $q'\in\mathbb{Z}_p\setminus\{q,-q\}$ to the equation $q'q\equiv 1\operatorname{mod}p$,
and such that after writing the Hirzebruch--Jung expansion
$$\frac{p}{q}=r_{1}-\frac{1}{r_{2}-\frac{1}{\cdots-\frac{1}{r_{k}}}},$$ we have that $r_{j}>2$ for $j=1,\ldots,k$.
(For example, $p=3$ and $q=1$ satisfy these conditions.) For such $p$ and $q$, the minimal model of $\widehat{C}_{p,\,q}$ will not contain any $(-2)$-curves
and there will be only one other lens space diffeomorphic to $L(p;\,q)$, namely $L(p;\,p-q)$. It follows that
the minimal resolution $\pi:(\widehat{M},\,\widehat{J}_{0})\to(\widehat{C}_{p,\,q},\,J_{0})$ is the unique (smooth) canonical model of $\widehat{C}_{p,\,q}$ \cite{ishii}.
Let $E$ denote the exceptional set of this resolution.
Since $1<\frac{p}{p-q}<2$, $2$ always appears in the above continued fraction expansion of $\frac{p}{p-q}$, and so $\widehat{C}_{p,\,p-q}$ does not admit a smooth canonical model. In particular, 
by \cite[Corollary B]{cds}, this complex cone never appears as the tangent cone of any smooth complete AC expanding gradient K\"ahler--Ricci soliton.

Let $(\widehat{M},\,\hat{g},\,X)$ be the unique AC expanding gradient 
K\"ahler--Ricci soliton given by \cite[Theorem A]{con-der} satisfying
\begin{equation}\label{lindsay}
\begin{split}
|(\nabla^{g_{0}})^k(\pi_{*}\hat{g}-g_{0})|_{g_0} &\leq C_{k}r^{-2-k}\quad\textrm{for all $k\in\mathbb{N}_{0}$ and $d\pi(X)=r\partial_{r}$}.
\end{split}
\end{equation}
(Note the paragraph above \cite[Claim 4.1]{cds} and \cite[Lemma 2.13]{cds} to see that \cite[Theorem A]{con-der} applies to our situation). 
Recall the isometry $\phi:(\widehat{C}_{p,\,q},\,g_{0})\to(\widehat{C}_{p,\,p-q},\,g_{0}),\,\phi(z_{1},\,z_{2})=(z_{1},\,\bar{z}_{2})$.
After rescaling the soliton and pulling back the cone metric by a suitable diffeomorphism (cf.~Section \ref{founders}), by abuse of notation
we can consider $(\widehat{M},\,\hat{g},\,X)$ an AC expanding gradient \emph{Ricci} soliton with
\begin{equation}\label{eq_pi_def}
\begin{split}
|(\nabla^{g_{0}})^k((\phi\circ\pi)_{*}\hat{g}-g_{0})|_{g_0} &\leq C_{k}r^{-2-k}\quad\textrm{for all $k\in\mathbb{N}_{0}$ and $[d(\phi\circ\pi)](X)=\frac{1}{2}r\partial_{r}$},
\end{split}
\end{equation}
i.e., we replace the original tangent cone
$(\widehat{C}_{p,\,q},\,g_{0})$ with the isometric non-biholomorphic tangent cone $(\widehat{C}_{p,\,p-q},\,g_{0})$.

Consider Type II deformations of $(\widehat{C}_{p,\,p-q},\,g_{0})$ defined by 
any smooth real-valued function $\varphi$ on $\widehat{C}_{p,\,p-q}\setminus\{0\}$ invariant under the 
diagonal $\mathbb{C}^{*}$-action and close enough to $0$ in the $C^\infty$-topology so that $\omega_{0,\,\varphi}:=\frac{i}{2}\partial\bar{\partial}(e^{2\varphi}r^{2})$ is positive-definite. Let $g_{0,\,\varphi}$ denote the corresponding K\"ahler cone metric on $\widehat{C}_{p,\,p-q}$ and let $\hat{r}:=re^{\varphi}$ denote the corresponding radial function. This metric is then invariant under the standard diagonal $S^{1}$-action on this cone
and recall from Section \ref{type2}
that the link of the cone remains unchanged. For $\varphi$ sufficiently small, we can apply the local theory of spaces of AC expanding Ricci solitons from \cite[Section 8]{bamler-chen} to obtain 
an expanding gradient Ricci soliton asymptotic to $g_{0,\varphi}$ without changing the diffeomorphism $\phi\circ\pi$ at infinity. More precisely, we have:
\begin{claim}\label{perturb}
There exists a complete expanding gradient Ricci soliton $\hat{g}_{\varphi}$ on $\widehat{M}$ with soliton vector field $X_\varphi$ satisfying $[d(\phi\circ\pi)](X_\varphi)=\frac{1}{2}\hat{r}\partial_{\hat{r}}$ for $\hat{r}$ sufficiently large, such that
\begin{equation}\label{richard}
|(\nabla^{g_{0,\,\varphi}})^k((\phi\circ\pi)_{*}\hat{g}_{\varphi}-g_{0,\,\varphi})|_{g_{0,\,\varphi}} \leq C_{k}\hat{r}^{-2-k}\quad\textrm{for all $k\in\mathbb{N}_{0}$}.
\end{equation}
\end{claim}

\begin{proof}[Proof of Claim \ref{perturb}]
To prove the claim, we will apply below the local description of the moduli space of AC expanding Ricci solitons near $(\widehat{M},\hat{g},X)$ from \cite[Section 8]{bamler-chen} together with a recent result on the kernel of the linearisation of the expanding Ricci soliton equation \eqref{hot} from \cite[Theorem 3]{ozuch-naff}. Our notation follows \cite{bamler-chen}.

Since $X$ is a gradient vector field, we can write $X=\widehat{\nabla}\hat{f}$ for some smooth real-valued function $\hat{f}:\widehat{M}\to\mathbb{R}$, where $\widehat{\nabla}$ denotes the Levi-Civita connection of $\hat{g}$. Let 
$$L_{\hat{g}}\coloneqq\triangle_{\hat{g}} -\widehat{\nabla}_{\widehat{\nabla} \hat{f}}+2\mathrm{Rm}_{\hat{g}}$$
denote the weighted Lichnerowicz operator associated to $(\widehat{M},\hat{g},X)$. Here, $\triangle_{\hat{g}}$ denotes the connection Laplacian of $\widehat{\nabla}$ and the action of $\Rm_{\hat{g}}$ on $(0,\,2)$-tensors is given by $\left(\Rm_{\hat{g}}(h)\right)_{ij}=(\Rm_{\hat{g}})\indices{_i^k^l_j}h_{kl}$ under the convention that $\Ric_{ij}=\Rm\indices{_i^k_k_j}$. This operator arises by considering the linearisation of \eqref{hot}  (see for instance \cite[Definition 1.1]{Der-Sta-Egs} or \cite[Section 2.5]{bamler-chen}).
For any $k\geq 0$ and $\alpha\in(0,\,1)$, we define a map $L_{\hat{g}}$ by
\begin{equation}\label{alix}
L_{\hat{g}}:C^{k+2,\,\alpha}_{-2,\,\widehat{\nabla}\hat{f}}(\widehat{M};S^2T^*\widehat{M})\rightarrow C^{k,\,\alpha}_{-2}(\widehat{M};S^2T^*\widehat{M})
\end{equation}
between symmetric $(0,\,2)$-tensors $S^2T^*\widehat{M}$, where the norms defining the weighted spaces are given by
\begin{alignat*}{2}
&\|u\|_{C^{k,\,\alpha}_{-a}}\coloneqq \left\|\left(\sup_{\widehat{M}}\hat{f}-\hat{f}+1\right)^{a/2}u\right\|_{C^{k,\alpha}_{\hat{g}}},\qquad&& \|u\|_{C^{k+2,\,\alpha}_{-a,\widehat{\nabla}\hat{f}}}\coloneqq\|u\|_{C^{k+2,\alpha}_{-a}}+\left\|\widehat{\nabla}_{\widehat{\nabla}\hat{f}} u\right\|_{C^{k,\alpha}_{-a}}, 
\\
& \|u\|_{H^k_{\hat{f}}}\coloneqq \sum_{j=0}^k \int|\widehat{\nabla}^j u|_{\hat{g}}^2\,e^{-\hat{f}}d\hat{g}, &&
\end{alignat*}
with $d\hat{g}$ denoting the volume form of $\hat{g}$. For a more detailed discussion of the norms and spaces involved, we refer the reader to \cite[Section 5]{bamler-chen}.

The fact that the maps $L_{\hat{g}}$ are well-defined is demonstrated in \cite[Proposition 5.33]{bamler-chen} (see also \cite[Corollary 2.4]{Der-Smo-Pos-Cur-Con}) and the fact that the kernels of all of these maps agree is proved in \cite[Lemma 8.11(b)]{bamler-chen}. 
We denote the kernel of \eqref{alix} by $K$. By \cite[Proposition 5.38]{bamler-chen}, we also know that $K\subset H^1_{\hat{f}}(\widehat{M};S^2T^*\widehat{M})$.

Since $(\widehat{M},\,\hat{g},\,X)$ is an expanding gradient \emph{K\"{a}hler}--Ricci soliton in the convention of \cite{ozuch-naff}, their Theorem $3$ asserts that $K\cap H^2_{\hat{f}}$ is trivial. Indeed, their proof demonstrates that $\langle L_{\hat{g}} h,h\rangle_{L^2_{\hat{f}}}\leq 0$ for all $h\in H^2_{\hat{f}}$, with equality only possible if $h\equiv 0$ or if $\hat{g}$ is a negative Einstein metric. The latter case never occurs in our setting because the curvature of $\hat{g}$ decays quadratically to zero at spatial infinity. Therefore, in order to check that $K$ is trivial, we will verify that $K\subset H^2_{\hat{f}}$ by checking that $\widehat{\nabla}^2\kappa\in L^2_{\hat{f}}$ (cf.~the proof of \cite[Theorem 5.38(b)]{bamler-chen}). Let $\eta\in C^\infty_c(\widehat{M};\mathbb{R})$ be a compactly supported cutoff function with $0\leq\eta\leq 1$ and $|\widehat{\nabla}\eta|_{\hat{g}}\leq 1$, and let $\kappa\in K\subset H^1_{\hat{f}}$. Integrating by parts, we see that 
\begin{align}
\int& \eta^2|\widehat{\nabla}^{ij} \kappa|_{\hat{g}}^2~e^{-\hat{f}}d\hat{g}=-2\int\eta\widehat{\nabla}^i\eta \widehat{\nabla}^j\kappa \widehat{\nabla}_{ij}\kappa \,e^{-\hat{f}}d\hat{g}-\int\eta^2\widehat{\nabla}^j\kappa\widehat{\nabla}\indices{^i_i_j}\kappa\,e^{-\hat{f}}d\hat{g}+\int\eta^2\widehat{\nabla}^j\kappa\widehat{\nabla}_{ij}\kappa\widehat{\nabla}^i\hat{f}\,e^{-\hat{f}}d\hat{g}\notag
\\
&= -2\int\eta\widehat{\nabla}^i\eta \widehat{\nabla}^j\kappa \widehat{\nabla}_{ij}\kappa \,e^{-\hat{f}}d\hat{g}-\int\eta^2\widehat{\nabla}^j\kappa\widehat{\nabla}\indices{^i_j_i}\kappa\,e^{-\hat{f}}d\hat{g}+\int\eta^2\widehat{\nabla}^j\kappa\widehat{\nabla}^i(\Rm_{\hat{g}}\ast\kappa)_{ij}\,e^{-\hat{f}}d\hat{g}\notag
\\
&\qquad+\int\eta^2\widehat{\nabla}^j\kappa\widehat{\nabla}_{ji}\kappa\widehat{\nabla}^i\hat{f}\,e^{-\hat{f}}d\hat{g}+\int\eta^2\widehat{\nabla}^j\kappa(\Rm_{\hat{g}}\ast\kappa)_{ij}\widehat{\nabla}^i \hat{f}\,e^{-\hat{f}}d\hat{g}\notag
\\
&=-2\int\eta\widehat{\nabla}^i\eta \widehat{\nabla}^j\kappa \widehat{\nabla}_{ij}\kappa \,e^{-\hat{f}}d\hat{g}-\int\eta^2\widehat{\nabla}^j\kappa\widehat{\nabla}_j(\triangle_{\hat{g}}\kappa -\widehat{\nabla}_{\widehat{\nabla} \hat{f}}\kappa)\,e^{-\hat{f}}d\hat{g}+\int\eta^2\widehat{\nabla}^j\kappa(\Rm_{\hat{g}}\ast\widehat{\nabla}\kappa)_{j}\,e^{-\hat{f}}d\hat{g}\notag
\\
&\qquad+\int\eta^2\widehat{\nabla}^j\kappa\widehat{\nabla}^i(\Rm_{\hat{g}}\ast\kappa)_{ij}\,e^{-\hat{f}}d\hat{g}-\int\eta^2\widehat{\nabla}^j\kappa\widehat{\nabla}_i\kappa\widehat{\nabla}\indices{_j^i}\hat{f}\,e^{-\hat{f}}d\hat{g}+\int\eta^2\widehat{\nabla}^j\kappa(\Rm_{\hat{g}}\ast\kappa)_{ij}\widehat{\nabla}^i \hat{f}\,e^{-\hat{f}}d\hat{g}\notag
\\
&\leq \frac{1}{2}\int\eta^2|\widehat{\nabla}^{ij}\kappa|_{\hat{g}}^2\,e^{-\hat{f}}d\hat{g}+C\int|\widehat{\nabla}\kappa|_{\hat{g}}^2\,e^{-\hat{f}}d\hat{g}+C\int|\kappa|_{\hat{g}}^2\,e^{-\hat{f}}d\hat{g}
\end{align}
for some constant $C>0$, where above we have suppressed the indices for the $(0,\,2)$-tensor field $\kappa$ and have used $\ast$ to denote (unspecified) contractions of pairs of tensors. To obtain the last inequality, we have used the boundedness of $|\widehat{\nabla}^2\hat{f}|_{\hat{g}}$ which is implied by \eqref{hot}, the soliton identity $R_{\hat{g}}+|\widehat{\nabla} \hat{f}|_{\hat{g}}^2-\hat{f}=\operatorname{const.}$, and the quadratic spatial decay of $\Rm_{\hat{g}}$ and of its derivatives. Letting the support of $\eta$ extend to all of $M$ then shows that $\kappa\in H^2_{\hat{f}}$, as claimed.

Next we consider, for $2\leq k^*\leq\infty$, the spaces 
$$\mathcal{M}^{k^*}\coloneqq\mathcal{M}^{k^*}\left(\widehat{M},L(p;\,p-q),\left(\phi\circ\pi|_{(\phi\circ\pi)^{-1}\left(\{r\geq1\}\right)}\right)^{-1}\right)$$
 from \cite[Definition 3.8]{bamler-chen}. Roughly, these are isometry classes of expanding Ricci solitons on $\widehat{M}$ asymptotic to cone metrics of regularity $C^{k^*}$ over the link $L(p;\,p-q)$ in the sense of \cite[Definition 3.8(3)]{bamler-chen}. Note that because of \eqref{eq_pi_def} and the fact that the cone metric $g_0$ is smooth, the triple $(\hat{g},\,X,\,g_0)$ represents an element of $\mathcal{M}^{k^*}$ for any $2\leq k^*\leq\infty$. We may then apply \cite[Proposition 8.34(b) and (c)]{bamler-chen} with (for instance) $k=30$ and $k^*=60$, and use the fact that $K$ is trivial to obtain the existence of an open neighborhood of the origin $U\subset C^{90}(\widehat{C}_{p,\,p-q},\,\mathbb{R})$ such that the following holds: for any $\varphi\in U$ invariant under the induced diagonal $\mathbb{C}^{*}$-action on $\widehat{C}_{p,\,p-q}$, there exists a class $p_\varphi\in\mathcal{M}^{60}$ comprising isometric expanding Ricci solitons on $\widehat{M}$ asymptotic to the cone $(\widehat{C}_{p,\,p-q},g_{0,\,\varphi})$. We may further impose that $\varphi$ is smooth so that $g_{0,\,\varphi}$ is a smooth cone metric. Then we again have by the definition of $\mathcal{M}^{k^*}$ that $p_\varphi\in\mathcal{M}^{\infty}$, so by \cite[Lemma 3.15(a)]{bamler-chen}, $p_\varphi$ has a ``$C^\infty$-regular representative'' $(\hat{g}_\varphi,X_\varphi,g_{0,\,\varphi})$ in the sense of \cite[Definition 3.8]{bamler-chen}. This means that $(\widehat{M},\hat{g}_\varphi,X_\varphi)$ is a complete expanding Ricci soliton with both $\hat{g}_\varphi$ and $X_\varphi$ smooth, that $[d(\phi\circ\pi)](X_\varphi)=\frac{1}{2}\hat{r}\partial_{\hat{r}}$, and that
  \begin{equation*}
\begin{split}
\hat{r}^{k}|(\nabla^{g_{0,\,\varphi}})^{k}((\phi\circ\pi)_{*}\hat{g}_{\varphi}-g_{0,\,\varphi})|_{g_{0,\,\varphi}}=o(1)\quad\textrm{for $k=0,\,1,\,2$.}
\end{split}
\end{equation*}
We now apply \cite[Lemmas 2.9 and 3.15(c)]{bamler-chen} to ascertain that \eqref{richard} holds for all $k\in\mathbb{N}_0$ with respect 
to the same map $\phi\circ\pi$ by \cite[Lemma 3.15(b)]{bamler-chen}. Finally, \cite[Proposition 6.1(iv)]{bamler-chen} implies that $X_\varphi$ is gradient.
\end{proof} 

The following claim will conclude the proof of Theorem \ref{thm_nonkahler}. 
\begin{claim}\label{ronanan}
The metric $\hat{g}_\varphi$ from \ref{perturb} is not K\"ahler with respect to any complex structure on $\widehat{M}$.
\end{claim}

\begin{proof}[Proof of Claim \ref{ronanan}]
Suppose to the contrary that $\hat{g}_\varphi$ is K\"ahler with respect to some complex structure $\widetilde{J}$ on $\widehat{M}$. Then being gradient, 
the soliton vector field $X_{\varphi}$ is real $\widetilde{J}$-holomorphic 
\cite[Section 2.2]{FIK} so that $(\widehat{M},\,\frac{1}{2}\hat{g}_{\varphi},\,2X_{\varphi})$ defines a complete AC expanding gradient K\"ahler--Ricci soliton
with respect to $\widetilde{J}$. The asymptotic cone is then necessarily a K\"ahler cone by \cite[Proposition 3.10]{cds}, the link of which must be diffeomorphic to 
the lens space $L(p;\,q)\cong L(p;\,p-q)$ by virtue of the asymptotics dictated by Claim \ref{perturb}. Note that by choice of $p$ and $q$, no other lens space is possible.    
By the classification theorem \cite[Theorem 8]{belgun}, we then see that the cone itself must be biholomorphic to either $\widehat{C}_{p,\,q}$
or $\widehat{C}_{p,\,p-q}$. Now, the structure theorem \cite[Corollary B]{cds} tells us that the asymptotic cone must admit a smooth canonical model.
Again, by choice of $p$ and $q$, this rules out $\widehat{C}_{p,\,p-q}$ as the asymptotic cone. Therefore the asymptotic
cone must be biholomorphic to $\widehat{C}_{p,\,q}$. \cite[Corollary B]{cds} again tells us that 
$(\widehat{M},\,\widetilde{J})$ is the smooth canonical model of $\widehat{C}_{p,\,q}$ and gives a
resolution map $\sigma:(\widehat{M},\,\widetilde{J})\to\widehat{C}_{p,\,q}$. 

Next, returning to the asymptotics given by \eqref{richard}, we see from observation \eqref{flowbaby} that
the Ricci flow $\hat{g}_{\varphi}(t)$ corresponding to $\hat{g}_{\varphi}$
satisfies $\lim_{t\to0^{+}}(\phi\circ\pi)_{*}\hat{g}_{\varphi}(t)=g_{0,\,\varphi}$
locally smoothly on the cone. Because $X_{\varphi}$ is real $\widetilde{J}$-holomorphic and $\hat{g}_{\varphi}$ is K\"ahler with respect to $\widetilde{J}$,  
$\hat{g}_{\varphi}(t)$ is K\"ahler with respect to $\widetilde{J}$ for all $t>0$, and so $0=\nabla^{\hat{g}_{\varphi}(t)}\widetilde{J}$ for all $t>0$.
The aforementioned convergence therefore gives us that
$$\nabla^{g_{0,\,\varphi}}[(\phi\circ\pi)_{*}\widetilde{J}]=\lim_{t\to0^{+}}\nabla^{(\phi\circ\pi)_{*}\hat{g}_{\varphi}(t)}(\phi\circ\pi)_{*}\widetilde{J}
=\lim_{t\to0^{+}}\nabla^{\hat{g}_{\varphi}(t)}\widetilde{J}=0,$$
i.e., $g_{0,\,\varphi}$ is K\"ahler with respect to the complex structure $J_{1}:=(\phi\circ\pi)_{*}\widetilde{J}$. By construction, 
$g_{0,\,\varphi}$ is also K\"ahler with respect to the standard complex structure $J_{0}$ on $\widehat{C}_{p,\,p-q}$.  

First, suppose that $J_{0}$ and $J_{1}$ induce the same orientation. Then the fact that 
$g_{0,\,\varphi}$ is not flat, and in particular not a four-dimensional hyper-K\"ahler cone, means that 
$J_{1}=\pm J_{0}$. In particular, $\phi\circ\pi:(\widehat{M},\,\widetilde{J})\to\widehat{C}_{p,\,p-q}$ is a holomorphic map. 
From above, we also have a resolution map $\sigma:(\widehat{M},\,\widetilde{J})\to\widehat{C}_{p,\,q}$. This yields a biholomorphism
$\phi\circ\pi\circ\sigma^{-1}:\widehat{C}_{p,\,q}\to\widehat{C}_{p,\,p-q}$ which is a contradiction because $\widehat{C}_{p,\,q}$ admits a smooth canonical model, whereas 
$\widehat{C}_{p,\,p-q}$ does not. $J_{0}$ and $J_{1}$ must therefore induce opposite orientations.

Next, suppose that this is the case. Viewing now $J_0$ and $J_1$ as endomorphisms of the real tangent bundle of the cone, let $A:=J_{0}\circ J_{1}$.
Since $J_{0}$ and $J_{1}$ induce opposite orientations by assumption and are both compatible with $g_{0,\varphi}$, they must commute. This implies that $A$ is symmetric. In addition, $A$ is orthogonal because both $J_{0}$ and $J_{1}$ are. The eigenvalues of $A$ are therefore
$\pm1$, and $A$ has determinant $1$ because the same is true for both $J_{0}$ and $J_{1}$. But $A\neq\pm I$ as $J_1\neq\pm J_0$ for orientation reasons, and so
we deduce that the eigenvalues of $A$ must be 
$1,\,1,\,-1,\,-1$. The (real) two-dimensional eigenspaces of $A$ corresponding to $1$ and $-1$ are parallel subspaces of the tangent space of the cone 
and are invariant under both $J_{0}$ and $J_{1}$ because $A$ commutes with both of these complex structures. 
These are the only two-dimensional subbundles of the tangent bundle of the 
cone invariant under \emph{both} $J_{0}$ and $J_{1}$ because any such subbundle must also be preserved by the endomorphism $A$.

Recall the radial vector field $\hat{r}\partial_{\hat{r}}$ of $g_{0,\,\varphi}$. Because 
$g_{0,\,\varphi}$ is a Riemannian cone metric, we know that $\operatorname{Ric}_{g_{0,\,\varphi}}(\hat{r}\partial_{\hat{r}})=0$, and so the nullspace of
$\operatorname{Ric}_{g_{0,\,\varphi}}$, viewed as an endomorphism of the real tangent bundle of the cone, is at least one-dimensional at every point.
Then because $g_{0,\,\varphi}$ is K\"ahler with respect to both $J_{0}$ and $J_{1}$, the nullspace of 
$\operatorname{Ric}_{g_{0,\,\varphi}}$ is invariant under the action of both $J_{0}$ and $J_{1}$, hence must either be four-dimensional, or
a two-dimensional eigenspace of $A$ corresponding to $1$ or $-1$. In the latter case, 
$\hat{r}\partial_{\hat{r}}$ must belong to the $1$ or $-1$ eigenspace of $A$ at every point, i.e., $A(\hat{r}\partial_{\hat{r}})=\pm\hat{r}\partial_{\hat{r}}$. 
Then for any vector $v$, we have that 
$A(\nabla^{g_{0,\,\varphi}}_{v}(\hat{r}\partial_{\hat{r}}))=\pm\nabla^{g_{0,\,\varphi}}_{v}(\hat{r}\partial_{\hat{r}})$ because $A$ is parallel, so that
$Av=\pm v$ because the endomorphism $v\mapsto\nabla^{g_{0,\,\varphi}}_{v}(\hat{r}\partial_{\hat{r}})$ of the tangent bundle of the cone is the identity
by virtue of the fact that $g_{0,\,\varphi}$ is a Riemannian cone. This means that $A=\pm I$, a contradiction, and so we conclude that 
the nullspace of $\operatorname{Ric}_{g_{0,\,\varphi}}$ is four-dimensional, i.e., the cone is Ricci-flat, and in particular globally flat. This yields another contradiction and completes the proof of the claim. 
\end{proof}

\bibliographystyle{amsalpha}

\bibliography{ref}

\end{document}